\documentclass [12 pt]{article}
 \usepackage{graphicx, amsthm, amsfonts, amsmath}
\usepackage{tikz}
\usepackage{color}
\usepackage{verbatim}
\title{Threshold Progressions in a Variety of Covering and Packing Contexts}
\author{Anant Godbole \\
East Tennessee State University \\
{\tt godbolea@etsu.edu} \and
Thomas Grubb \\
University of California, San Diego \\
{\tt tgrubb@ucsd.edu}\and
Kyutae Han\\
University of California, Los Angeles\\
{\tt kyutae.paul.han@math.ucla.edu}\and
Bill Kay\\
Oak Ridge National Laboratories\\
{\tt kaybw@ornl.gov}}
\begin{document}
\def\tv{d_{\rm TV}}
\def\P{\rm Poi}
\def\l{\lambda}
\def\p{\mathbb P}
\def\v{\mathbb V}
\def\a{\alpha}
\def\e{{\mathbb E}}
\def\lr{\left(}
\def\rr{\right)}
\def\lc{\left\{}
\def\rc{\right\}}
\def\cf{{\cal F}}
\def\cc{{\cal C}}
\def\ca{{\cal A}}
\def\cb{{\cal B}}
\def\cl{{\cal L}}
\def\bfx{\bf x}
\def\nkt{{n-t}\choose{k-t}}
\def\nt{{n \choose t}}
\newcommand{\marginal}[1]{\marginpar{\raggedright\scriptsize #1}}
\newcommand\numberthis{\addtocounter{equation}{1}\tag{\theequation}}
\newcommand{\x}{\mathbf{x}}
\newcommand{\y}{\mathbf{y}}
\newcommand{\red}{\color{red}}
\newcommand{\blue}{\color{blue}}
\newcommand{\beq}{\begin{equation}}
\newcommand{\eeq}{\end{equation}}
\newcommand{\ba}{\mathbf{a}}
\newcommand{\bb}{\mathbf{b}}
\newcommand{\bc}{\mathbf{c}}
\newcommand{\bd}{\mathbf{d}}
\newcommand{\be}{\mathbf{e}}
\newcommand{\boldf}{\mathbf{f}}
\newcommand{\bg}{\mathbf{g}}
\newcommand{\bh}{\mathbf{h}}
\newcommand{\bi}{\mathbf{i}}
\newcommand{\bj}{\mathbf{j}}
\newcommand{\bk}{\mathbf{k}}
\newcommand{\bl}{\mathbf{l}}
\newcommand{\bm}{\mathbf{m}}
\newcommand{\bn}{\mathbf{n}}
\newcommand{\bo}{\mathbf{o}}
\newcommand{\bp}{\mathbf{p}}
\newcommand{\bq}{\mathbf{q}}
\newcommand{\br}{\mathbf{r}}
\newcommand{\bs}{\mathbf{s}}
\newcommand{\bt}{\mathbf{t}}
\newcommand{\bu}{\mathbf{u}}
\newcommand{\bv}{\mathbf{v}}
\newcommand{\bw}{\mathbf{w}}
\newcommand{\bx}{\mathbf{x}}
\newcommand{\by}{\mathbf{y}}
\newcommand{\bz}{\mathbf{z}}
\def\bfa{({\bf a}_1,\ldots,{\bf a}_{g+1})}
\newtheorem{thm}{Theorem}[section]
\newtheorem{cor}[thm]{Corollary}
\newtheorem{lem}[thm]{Lemma}
\newtheorem{claim}[thm]{Claim}
\newtheorem{prop}[thm]{Proposition}
\newtheorem{conj}[thm]{Conjecture}
\newtheorem{fact}[thm]{Fact}
\maketitle
\let\thefootnote\relax\footnote{Notice: This manuscript has been authored in part by UT-Battelle, LLC under Contract No. DE-AC05-00OR22725 with the U.S. Department of Energy. The United States Government retains and the publisher, by accepting the article for publication, acknowledges that the United States Government retains a non-exclusive, paid-up, irrevocable, world-wide license to publish or reproduce the published form of this manuscript, or allow others to do so, for United States Government purposes. The Department of Energy will provide public access to these results of federally sponsored research in accordance with the DOE Public Access Plan (http://energy.gov/downloads/doe-public-access-plan).}
\newpage
\begin{abstract} Using standard methods (due to Janson, Stein-Chen, and Talagrand) from probabilistic combinatorics, we explore the following
general theme: As one progresses from each member of a family of objects
$\ca$ being ``covered" by at most one object in a random collection $\cc$, to being
covered at most $\lambda$ times, to being covered at least once, to being covered at least $\lambda$ times, a hierarchy of thresholds emerge.    
  We will then see how such results vary according to the context, and level of dependence introduced.  
Examples will be from extremal set theory; combinatorics; and additive number theory.

\end{abstract}

\section{Introduction and Motivating Example} 
Suppose $\ca$ is a family of combinatorial objects which can be covered (in some sense) by members of a family $\mathcal{G}$. Suppose further that $\cc\subseteq \mathcal{G}$ is a random subset of $\mathcal{G}$ in which each element of $\mathcal{G}$ is selected for membership in $\cc$ with probability $p$. We examine the following in a variety of contexts:

For which regimes of $p$ do we have with high probability (whp) or with low probability (wlp)?
\begin{enumerate}
\item Each member of $\ca$ is covered by at most one member of $\cc$ (i.e., $\ca$ packs into $\cc$). 
\item Each member of $\ca$ is covered by at most $\lambda$ members of $\cc$ (i.e., $\ca$ $\lambda$-packs into $\cc$).
\item Each member of $\ca$ is covered by at least one member of $\cc$ (i.e. $\ca$ is covered by $\cc$).
\item Each member of $\ca$ is covered by at least $\lambda$ members of $\cc$ (i.e., is $\lambda$-covered by $\cc$). 
\end{enumerate}
Further, in which contexts can we establish a sharp probabilistic threshold?

In this section, we introduce results from the classical theory of the random allocation of balls to boxes.  We then see how and to what extent the results apply to situations such as coverage of sets by other sets (Section 2); of integers by $h$-sets of integers (Section 3); of $n$-permutations by $(n+1)$-permutations (Section 4); and of sets via unions of other sets (Section 5).  It is our hope that the paradigm that emerges will be explored by other researchers in a variety of other contexts.

Suppose that we are trying to pack balls in boxes so that each box contains at most one ball.  This is the so-called ``birthday problem", and it is well-known, e.g.,  \cite{bh}, that if we throw $n$ balls into $N$ boxes uniformly at random, then the threshold for the property to hold whp/wlp is $n=\sqrt{N}$. 
  Throughout this paper, we will usually not mention behavior at the threshold, but this can be derived in most cases.  The Stein-Chen method (\cite{bhj}) is often used to exhibit the threshold $n=N^{\lambda/(\lambda+1)}$ for the property ``each box contains at most $\lambda$ balls", but we rederive this next using    {\em Talagrand's inequality}, Theorem 7.7.1 in \cite{as}.

\begin{thm}[Talagrand's Inequality \cite{as}]\label{Talagrand}
Let $X\ge0$ be determined by $n$ random trials.  We say that $X$ is {\em Lipschitz} if changing the outcome of any one trial affects the value of $X$ by at most $1$. We say that $X$ is {\em $f$-certifiable} if the event $\{X \geq s\}$ can be verified by revealing the outcomes of $f(s)$ trials. Given an $f$-certifiable Lipschitz $X$, for all $b$, $t$, we have:

\[
\p[X \leq b-t\sqrt{f(b)}]\p[X \geq b] \leq e^{-t^2/4}.
\]
\end{thm}

\begin{thm} \label{Tballs} When $n$ balls are randomly and uniformly distributed in $N$ boxes, then letting $X=X_\lambda$ denote the number of boxes with $\lambda+1$ or more balls,

$$n\ll N^{\lambda/(\lambda+1)}\Rightarrow \p(X=0)\to1,$$
and 
$$n\gg N^{\lambda/(\lambda+1)}\Rightarrow \p(X=0)\to0,$$
where throughout the paper, given $f(n), g(n)\ge0$, we write $f(n)\ll g(n)$ (or $g(n)\gg f(n)$) if $f(n)=o(g(n))$ as $n\to\infty.$
\end{thm}

\begin{proof}  The first half is routine and follows from Markov's inequality and  the fact that if $n\ll N^{\lambda/(\lambda+1)}$,
\[\p(X\ge 1)\le\e(X)\le N{n\choose {\lambda+1}}\lr\frac{1}{N}\rr^{\lambda+1}\to0,\]  and hence $\p(X=0) \rightarrow 1$, as claimed. We seek to employ Talagrand's inequality for the second half.  In this case,  altering the location of any one ball can only affect $X$ by at most one (e.g, $X$ is $1$-Lipschitz). Moreover, the event $\{X\ge s\}$ can be certified by the outomes of $s(\lambda+1)$ trials (e.g., $X$ is $s(\lambda+1)$-certifiable), so that taking $b = {\rm Med } (X)$ and $t = \frac{\sqrt{{\rm Med}(X)}}{\sqrt{\lambda + 1}}$ in  Talagrand's inequality  yields
\[\p(X=0)\le 2\exp\lc-\frac{{\rm Med}(X)}{4(\lambda+1)}\rc,\]
where ${\rm Med}(X)$ is a median of $X$.  Since the median and mean of $X$ differ by at most $40\sqrt{(\lambda+1)\e(X)}$ as per Fact 10.1 in \cite{mr}, we see that $\p(X=0)\to0$ whenever $\e(X)\to\infty$. Noting that $X=\sum_{j=1}^NI_j$ where $I_j$ is the indicator variable of the event that the $j$th box has $\geq \l+1$ balls in it, we see that if $ N^{\lambda/(\lambda+1)} \ll n \ll N$, then:


\begin{align*}
\e(X) &=N\p(I_1=1)\\&\geq N\binom{n}{\lambda+1}\left(\frac{1}{N}\right)^{\lambda+1} \left(1- \frac{1}{N}\right)^{n-\lambda-1}\\
&\geq N\frac{(n-\lambda -1)^{\lambda +1}}{(\lambda+1)!}\frac{1}{N^{\lambda+1}} \exp\lc{-\frac{n-\lambda -1}{N} (1+o(1))}\rc\\
&\rightarrow \infty
\end{align*} 
and hence $\p(X=0) \to 0$ as desired. Moreover, by monotonicity $\e(X)$ is a non-decreasing function of $n$, and so if $n $ is even larger, we have $\e(X) \rightarrow \infty$  and hence $\p(X=0) \to 0$, as desired. \hfill \end{proof}


Note that the thresholds in Theorem 1.1 get close to $n=N$ as $\lambda\to\infty$.
It may still be the case, however, that not all boxes will have a ball in them if $n\gg N$, which leads us to the covering questions.  It is well known (see, e.g., \cite{bh}) that the expected waiting time for each of the boxes to be covered by at least one ball is $N(\ln N+\gamma+o(1)),$ where $\gamma$ is Euler's constant, and that the variance of the waiting time is $\Theta( N^2)$.
Various people, e.g., \cite{h}  have asked about covering each box $\lambda$ or more times. 
Generalizing work of Erd\H os and R\'enyi; and Newman and Shepp, Holst \cite{h} produced the following definitive result:
\begin{thm} (Holst \cite{h})\label{Holst}:  Let $X=X_\lambda$ denote the waiting time until each box has at least $\lambda$ balls. We have: $$\e(X)=N(\ln N+(\lambda-1)
\ln \ln N+\gamma-\ln (\lambda-1)!+o(1)).$$Normalizing by setting
$X^\ast=X/N-\ln N-(\lambda-1)\ln\ln N+\ln(\lambda-1)!$, we have that
$X_1,\ldots, X_\lambda$ are asymptotically independent.  Moreover $$\p(X^\ast \le u)\to\exp\{-e^{-u}\}.$$  
\end{thm}
Theorem~\ref{Holst} implies the following threshold result:

\begin{cor}\label{thresh:Holst}
Let $r = r(n)$ be arbitrary. Let $X=X_\lambda$ denote the waiting time until each box has at least $\lambda$ balls. We have:

\[
r \rightarrow +\infty \Rightarrow \p(X\le N\lc\ln N+(\lambda-1)\ln \ln N+r\rc)\to 1,
\]
and
\[
r \rightarrow -\infty \Rightarrow 
\p(X\le N\lc\ln N+(\lambda-1)\ln \ln N+r\rc)\to 0.\]

\end{cor}

\begin{proof}
Let $X^\ast$ be as in Theorem~\ref{Holst}. Then we have:
\begin{align*}
\p(X \leq N(\ln n + (\lambda-1) \ln \ln n + r)) & = \p(X^\ast \leq r+O(1))\\
& \rightarrow \exp \lc -e^{-(r+O(1))} \rc,
\end{align*}
which establishes the desired result.\hfill\end{proof}

Of particular note is the linearity (in $\ln\ln N$) for coverings beyond the first, showing that an additional iterated logarithmic fraction suffices for each subsequent covering (which are asymptotically independent!)  We shall show that many of these features stay intact even as dependence is introduced into the covering context.  The main results of this paper are Theorems 2.3 and 2.4 on combinatorial designs; Theorem 3.4 on Sidon sets; Theorems 4.5 and 4.6 on permutations; and Theorem 5.1 on weakly union free set systems.  We emphasize that while our proofs use techniques that are now considered a standard part of the random combinatorial methods toolkit, there is a great deal of variety in which the ``balls in boxes paradigm" is shown to be valid, further begging some of the Open Questions mentioned in Section 6.

\section {Combinatorial Designs} What is the smallest number $C_1(n,k,t)$ of $k$-sets of $[n]$ that must be picked so that each $t$-set is contained in at at least one $k$-set?  In this area, extremal behavior has been well-studied:  The Erd\H os-Hanani Conjecture, its first proof by R\"odl (see \cite{as}), and the branching processes/greedy random algorithm proof of Spencer \cite{s} are all well-known.  The result is that the obvious lower bound of 
$$C_1(n,k,t)\ge\frac{{n\choose t}}{{k\choose t}}=\frac{{n\choose k}}{{{n-t}\choose {k-t}}}$$ is asymptotically correct as $n\to\infty$ with $k,t$ being held fixed.  These statements get reversed if we study the packing problem of having each $t$ set being contained in at most one $k$-set.  
We also have the general upper bound of Erd\H os and Spencer \cite{es}:  
$$C_1(n,k,t)\le \frac{{n\choose t}}{{k\choose t}}\lr1+\ln {k\choose t}\rr,$$
and the generalization from \cite{gtv}, which states that  $$C_\lambda(n,k,t)\le \frac{{n\choose t}}{{k\choose t}}\lr\ln(\lambda-1)!+\ln {k\choose t}+(\lambda-1)\ln\ln{k\choose t}\rr,$$ where $C_\lambda(n,k,t)$ is the smallest number of of $k$-sets of $[n]$ that must be picked so that each $t$-set is contained in at least $\lambda$ $k$-sets.

  Theorem~\ref{Tballs} was stated as a threshold result with respect to the cardinality of selected sets. As in \cite{gj} and \cite{gjlr}, frequently it is convenient to restate such results as probabilistic thresholds, which we will do throughout this paper, remarking that the so-called cardinality threshold results often go hand in hand.

We have the following result from \cite{gj}:

\begin{thm}\label{TDesign}
Let $r = r(n)$ be arbitrary. Let $\cc\subseteq \binom{[n]}{k}$ be a random subset of $\binom{[n]}{k}$ in which each element of $\binom{[n]}{k}$ is selected for membership in $\cc$ with probability $p:=\frac{1}{\binom{n-t}{k-t}} \left (\ln \binom{n}{t} +r \right)  $. Let $X$ denote the number of elements of $\binom{[n]}{t}$ which are not subsets of any member of $\cc$.   We have:

\[
r \rightarrow + \infty  \Rightarrow \p(X = 0) \to 1,
\]
and
\[ 
r \rightarrow - \infty\Rightarrow \p(X = 0) \to 0.\]

\end{thm}


In this section we offer two results  which complement Theorem~\ref{TDesign}; a $\lambda$-packing threshold (Theorem~\ref{PDesign}) and a $\lambda$-covering threshold (Theorem~\ref{TGDesign}). First, however, we mention the following  simplified version of Lemma A.2.5 in \cite{bhj}, which deals with tail sums of the binomial distribution, and which will be used frequently through the paper.   

\begin{lem}\label{tails}
Let $p = p_n$ be arbitrary. We have:

\[
np  \rightarrow 0   \Rightarrow \sum_{j = t_0}^{t_1} \binom{n}{j} p^j(1-p)^{n-j} =\binom{n}{t_0} p^{t_0}(1-p)^{n-t_0}(1+o(1)),
\]
for any fixed $0 \leq t_0 < t_1 \leq n$.   Moreover, for $0 \leq t_0 < t_1 =O(1)$, 
\[ 
np  \rightarrow \infty   \Rightarrow \sum_{j = t_0}^{t_1} \binom{n}{j} p^j(1-p)^{n-j} =\binom{n}{t_1} p^{t_1}(1-p)^{n-t_1}(1+o(1)) .\]

\end{lem}

Informally, Lemma~\ref{tails} provides conditions under which a cumulative binomial sum can be well-approximated by its first (or last) included term.  We are now ready to state and prove Theorem~\ref{PDesign}

\begin{thm}\label{PDesign}
Let $\lambda \geq 1$. Let $\cc\subseteq \binom{[n]}{k}$ be a random subset of $\binom{[n]}{k}$ in which each element of $\binom{[n]}{k}$ is selected for membership in $\cc$ with probability $p$. Let $X= X_\lambda$ denote the number of elements of $\binom{[n]}{t}$ which are  subsets of at least $\lambda+1$ members  of $\cc$.  We have:
 
\[
p  \ll \frac{1}{n^{(k-t) + t/(\lambda+1)}} \Rightarrow \p(X = 0) \to 1,
\]
and
\[ 
p  \gg \frac{1}{n^{(k-t) + t/(\lambda+1)}} \Rightarrow \p(X = 0) \to 0.\]

\end{thm}

\begin{proof}

By Markov's inequality and Lemma~\ref{tails}, we have:

\begin{eqnarray}\p(X\ge 1)&\le& \e(X)={n\choose t}\sum_{r=\lambda+1}^{{n-t}\choose{k-t}}{{\nkt}\choose r}p^r(1-p)^{{\nkt}-r}\nonumber\\
&\le& {n\choose t}{{{{n}\choose{k-t}}}\choose{\lambda+1}}p^{\lambda+1}(1-p)^{{\nkt}-\lambda-1}(1+o(1))\nonumber\\ 
&\le& \Lambda_{k,t,\lambda}{n^t}n^{(k-t)(\lambda+1)}p^{\lambda+1}.\end{eqnarray}
where $\Lambda_{k,t,\lambda}$ is a constant. Hence, if $p \ll \frac{1}{n^{(k-t) + t/(\lambda+1)}}$, we have $\p(X \ge 1) \rightarrow 0,$ proving the first part. 
 
  For the second part, we seek to employ Talagrand's inequality. $X$ is $\binom{k}{t}$-Lipschitz, as reflipping the coin to determine membership of any $k$-set in $\cc$ can affect the value of $X$ by at most $\binom{k}{t}$. Moreover, $X$ is $s(\lambda+1$)-certifiable as the event $\{X\ge s\}$ can be certified by the outomes of $s(\lambda+1)$ trials, so that (as in the proof of Theorem~\ref{Holst}), if $\e(X) \rightarrow \infty$, we have $\p(X = 0) \rightarrow 0$. Applying standard inequalities for the expression we have derived for $\e(X)$, we see that
  
  \[ 
  \e(X)\ge\Gamma_{k,t,\lambda}{n^t}n^{(k-t)(\lambda+1)}p^{\lambda+1}\exp\lr-\frac{p}{1+o(1)}{\nkt}\rr(1-o(1)),
  \]
where $\Gamma_{k,t,\lambda}$ is constant. Assume $\frac{1}{n^{(k-t) + t/(\lambda+1)}}\ll p \ll \frac{1}{n^{k-t}}$ so that $p\binom{n-t}{k-t} \rightarrow 0$. For this choice of $p$, $\e(X) \rightarrow \infty$. As in the proof of Theorem~\ref{Tballs} we are done by monotonicity. \hfill
 \end{proof}

Theorem~\ref{TGDesign} extends Theorem~\ref{TDesign} to $\lambda$-coverings for $\l\ge2$:

\begin{thm}\label{TGDesign}
Let $r = r(n)$ be arbitrary. Let $\cc\subseteq \binom{[n]}{k}$ be a random subset of $\binom{[n]}{k}$ in which each element of $\binom{[n]}{k}$ is selected for membership in $\cc$ with probability $p:=\frac{1}{\binom{n-t}{k-t}} \left (\ln \binom{n}{t} +(\lambda -1) \ln \ln \binom{n}{t} +r \right)  $. Let $X$ denote the number of elements of $\binom{[n]}{t}$ which are subsets of at most $\lambda-1$ members of $\cc$.   Then,

\[
r \rightarrow + \infty  \Rightarrow \p(X = 0) \to 1,
\]
and
\[ 
r \rightarrow - \infty\Rightarrow \p(X = 0) \to 0.\]
\end{thm}

We will prove Theorem~\ref{TGDesign} via the Stein-Chen method, using the following result from \cite{bhj}.

\begin{lem}(\cite{bhj})\label{SC}
Let $\{I_i\}_{i=1}^n$ be a collection of indicator random variables. Suppose that for each $j \in [n]$ there exists a sequence of random variables $\{J_{j,i}\}_{i=1}^n$ on the same probability space with:

\[
\cl\lr J_{j,1}, \ldots, J_{j,n}\rr =\cl\lr I_1, \ldots, I_n \vert I_j = 1  \rr,
\]
where $\cl(Z)$ denotes the distribution of $Z$. We call such a collection a {\em coupling}. Let $X = \sum_{i=1}^n I_i$, and let $\mu = \e(X)$.  Then, with $\P(\mu)$ denoting the Poisson distribution with mean $\mu$,

\begin{enumerate}
\item If $J_{j,i} \leq I_i \ \forall\  i \in [n]\setminus\{j\}$ (i.e., the $\{I_i\}_{i=1}^n$ are {\em negatively related}), 

\[
\tv(\cl(X), \P(\mu)) \leq (1 - e^{-\mu})\left(1-\frac{\v(X)}{\mu}\right).
\] 

\item If $J_{j,i} \geq I_i \  \forall \ i \in [n]\setminus\{j\}$ (i.e., the $\{I_i\}_{i=1}^n$ are {\em positively related}), 

\[
\tv(\cl(X), \P(\mu)) \leq \frac{1-e^{-\mu}}{\mu} \left(\v(X) -\mu +2\sum_{i=1}^n \p^2(I_i =1)\right).
\]
\end{enumerate}

\end{lem} 
\noindent {\it Proof of Theorem 2.4.}
We have $X=\sum_{j=1}^{n\choose t}I_j,$
where $I_j=1$  if the $j$th $t$-set is covered $\lambda-1$ or fewer times ($I_j=0$ otherwise).  
  We next (partially) exhibit the coupling from Lemma 2.5:  If $I_j=1$, i.e., if the $j$th $t$-set is covered by at most $\lambda-1$ $k$-sets, we let $J_{ji}=I_i$ for each $i$.  On the other hand, if the $j$th $t$-set is covered $\lambda$ or more times, we deselect a certain number of $k$-sets (according to the appropriate distribution) so as to achieve a sample outcome corresponding to $I_j=1$.  We then set $J_{ji}=1$ if the $i$th $t$ set is covered $\lambda-1$ or fewer times after this is done.  Since the conditional distribution is attained by a process of deselection, we  must have $J_{ji}\ge I_i$ for each $i\ne j$ (since a set that is covered at most $ \lambda-1$ times cannot be  covered at least $ \lambda$ times after some $k$-sets are deselected), so that the indicators $I$ are positively related, and we get via Lemma~\ref{SC} that
\begin{eqnarray}\tv(\cl(X),\P(\mu))&\le&\frac{1-e^{-\mu}}{\mu}\lr\v(X)-\mu+2\sum\p^2(I_j=1)\rr\nonumber\\
&\le&\p(I_1=1)+\frac{1}{\mu}\lr\sum_{i\ne j}[\e(I_iI_j)-\e(I_i)\e(I_j)]\rr\nonumber\\
&\le&\frac{\sum_{j\ne 1}\e(I_1I_j)}{\p(I_1=1)}-(N-2)\p(I_1=1),\end{eqnarray}
where $N=\nt$.  To calculate $\rho=\sum_{j\ne 1}\e(I_1I_j)$, consider the case that the $1$st and $j$th $t$-sets have an intersection of size $r$.  Letting $\rho_r$ equal $\e(I_1I_j)$ for such sets, so that $\rho=\sum_r{t\choose r}{{n-t}\choose{t-r}}\rho_r$, we have
\begin{equation}
\rho_r
=\sum_{s\le \lambda-1}\sum_{u\le \lambda-1}\sum_{v\le \min{\{s,u\}}}\rho_{r,s,u,v},\end{equation}
where $\rho_{r,s,u,v}$ is the probability that two $t$-sets that overlap in $r$ elements are both covered by $v$ $k$-sets, and individually by a total of $s$ and $u$ sets.  With 
$$M={\nkt}$$ and 
\[R:={{n-2t+r}\choose{k-2t+r}}\le{{n-t-1}\choose{k-t-1}}=:P\le M,\] we see that 

\begin{eqnarray}\rho_{r,s,u,v}&=&{R\choose v}{{M-R}\choose{s-v}}{{M-R}\choose {u-v}}p^{(s-v)+(u-v)+v}(1-p)^{2M-R-(s+u-v)}\nonumber\\
&=&{R\choose v}{{M-R}\choose{s-v}}{{M-R}\choose {u-v}}p^{(s-v)+(u-v)+v}(1-p)^{2M-R}(1+o(1)).\nonumber\\
\end{eqnarray} Now the quantity ${{M-R}\choose{x}}p^x$ is increasing in $x$ since we may assume without loss that $Mp=\omega(1)$, so that the sum in (3) is dominated by the $s=u=\lambda-1$ terms and thus 
\begin{equation}\rho_r=\sum_{v\le \lambda-1}{R\choose v}{{M-R}\choose {\lambda-1-v}}^2p^{2\lambda-2-v}(1-p)^{2M-R}(1+o(1)).\end{equation}  Consider the summand in (5).  We have:
\[{R\choose v}= O\left( n^{(k-2t+r)v}\right);\]
\[{{M-R}\choose{\lambda-v-1}}^2= O\left(n^{2(k-t)(\lambda-1-v)}\right) ;\]and assuming without loss of generality (again, by monotonicity) that 
\[p^{2\lambda-2-v}\le A_{n,k,t}\frac{\ln^{2\lambda-2-v} n} {n^{(k-t)(2\lambda-2-v)}},\]
we have that \[\rho_r=O\lr\ln^{2\lambda-2}n\sum_v n^{v(r-t)}(1-p)^{2M-R}\rr\]
is dominated by its $v=0$ term, as $r < t$.   Returning to (5), we see thus that 
\[\rho_r={{M-R}\choose{\lambda-1}}^2p^{2\lambda-2}(1-p)^{2M-R}(1+o(1)),\] so that
\begin{eqnarray}\rho&=&\sum_r{t\choose r}{{n-t}\choose{t-r}}{{M-R}\choose{\lambda-1}}^2p^{2\lambda-2}(1-p)^{2M-R}(1+o(1))\nonumber\\
&\le&{n\choose t}\max_r{{M-R}\choose{\lambda-1}}^2p^{2\lambda-2}(1-p)^{2M-R},\end{eqnarray}
and hence by (2),
\begin{eqnarray}
\tv(\cl(X),\P(\mu))&\le&\frac{{n\choose t}\max_r{{M-R}\choose{\lambda-1}}^2p^{2\lambda-2}(1-p)^{2M-R}}{{{M}\choose{\lambda-1}}p^{\lambda-1}(1-p)^{M-\lambda+1}}(1+o(1))\nonumber
\\&&\qquad -{n\choose t}{{{M}\choose{\lambda-1}}p^{\lambda-1}(1-p)^{M-\lambda+1}}(1+o(1))\nonumber\\
&=&\mu\lr \max_r \frac{{{M-R}\choose{\lambda-1}}^2}{{{M}\choose{\lambda-1}}^2}(1-p)^{2\lambda-2-R}-1\rr\nonumber\\
&=&\max_r \mu pR(1+o(1))\nonumber\\
&\le&\frac{B_{n,k,t}\mu\ln n}{n},
\end{eqnarray}
assuming that $p=O(\ln n/n^{k-t})$.  We thus have that the total variation distance tends to $0$ as long as $\mu$ is not too large.  In particular, there exists $\epsilon_n = o(1)$ so that 
\[e^{-\mu}-\epsilon_n\le \p(X=0)\le e^{-\mu}+\epsilon_n,\]  holds, and (by monotonicity) $\p(X=0)$ tends to $0$ or $1$ whenever $\e(X)$ tends to $\infty$ or 0 respectively.  All that remains is to figure out when this occurs.  Since $p \binom{n-t}{k-t} \rightarrow \infty$, we can apply Lemma~\ref{tails} to our computation of $\e(X)$ to see:

\begin{eqnarray*}
\e(X)&=&{n\choose t}\sum_{j=0}^{\lambda-1}{{\nkt}\choose{j}}p^j(1-p)^{{\nkt}-j}\\
&=&{n\choose t}{{\nkt}\choose{\lambda-1}}p^{\lambda-1}(1-p)^{{\nkt}-\lambda+1}(1+o(1))\\
&=&{n\choose t}\frac{{\nkt}^{\lambda-1}}{(\lambda-1)!}p^{\lambda-1}e^{-p{\nkt}}(1+o(1)).
\end{eqnarray*}

Plugging in $p$ as in the statement of the Theorem yields the desired results.  Specifically, we see that for any constant $K$, $$p:=\frac{1}{\binom{n-t}{k-t}} \left (\ln \binom{n}{t} +(\lambda -1) \ln \ln \binom{n}{t} +K\right)$$ gives that $\e(X)=\Theta(1)$.
\hfill\qed

\section {Sidon Sets and Additive Bases} 

A set $\ca\subseteq[n]$ is said to be a $B_h$ set (the totality of these for all $h\ge 2$ are known as Sidon sets) if each of the ${{\vert\ca\vert+h-1}\choose {h}}$ sums of elements drawn with replacement from $\ca$ are distinct.  A set $\ca\subseteq[n]\cup\{0\}$ is said to be an $h$-additive basis if each $j\in [n]$ can be written as the sum of $h$ elements in $\ca$.  Thus, a set is $h$-Sidon or an $h$-additive basis if each element in the potential sumset can be obtained in at most one or at least one way using elements of $\ca$.  We clearly thus have a packing/covering analogy, but as in previous sections, we will not use the word ``packing".  It is known that maximal Sidon sets and minimal additive bases are both of order $n^{1/h}$; for example minimal 2-additive bases have size $1.463\sqrt{n}\le\vert\ca\vert\le1.871\sqrt{n}$. See \cite{gjlr} and \cite{gllt} for details.  

We are interested, however, in random versions of these results, and three basic facts along these lines are as follows:

\begin{thm} (\cite{gjlr}) Consider a subset $\ca=\ca_n$ of size $k_n$ chosen
at random from the ${{n}\choose{k_n}}$ such subsets of $[n]$.
Then for any $h\ge2$,
$$k_n=o(n^{1/2h})\Rightarrow\p(\ca_n\ {\rm is\ }
B_h)\to1\quad(n\to\infty)$$
and
$$n^{1/2h}=o(k_n)\Rightarrow\p(\ca_n\ {\rm is\ }
B_h)\to0\quad(n\to\infty).$$
\end{thm}
We say that $\ca$ is an $\alpha$-truncated $h$-basis, if each element of $[\alpha n,(h-\alpha)n]$ can be expressed as an $h$-sum of elements in $\ca$.  
\begin{thm} (\cite{gllt})  
For $h\ge 2$, if we choose elements of $\{0\}\cup[n]$ to be in $\ca$ with probability  
$$p=\sqrt[h]{\frac{K\log n-K \log{\log{n}}+A_n}{n^{h-1}}},$$ where $K=K_{\a,h}=\frac{h!(h-1)!}{\a^{h-1}}$, 
then
$$\p(\ca\ {\it is\ an}\ \alpha-{\it truncated}\ h-{\it basis})\rightarrow\begin{cases} 0 & \mbox{if}\ A_n\rightarrow-\infty\\  1 & \mbox{if}\ A_n\rightarrow \infty \\ \exp\{- \frac{2\alpha}{h-1}e^{-A/K}\} & \mbox{if}\ A_n\rightarrow A\in{\mathbb R}\end{cases}.$$
\end{thm}
The case $h=2$ is studied in greater detail in the next result, which addresses coverage of each sum $g$ times.  (For historical reasons, we use $g$ in the place of $\lambda$ when studying Sidon sets.)
\begin{thm} (\cite{ghk}) If we choose elements of $\{0\}\cup[n]$ to be in $\ca$ with probability  
$$p=\sqrt{\frac{\frac{2}{\a}\log n+(g-2)\frac{2}{\a} \log{\log{n}}+A_n}{n}},$$ then
$$\p(\ca\ {\it is\ an}\ \alpha-{\it truncated}\ (2-g)-{\it basis})\rightarrow\begin{cases} 0 & \mbox{if}\ A_n\rightarrow-\infty\\  1 & \mbox{if}\ A_n\rightarrow \infty \\ \exp\{- {2\alpha}e^{-A\a/2}\} & \mbox{if}\ A_n\rightarrow A\in{\mathbb R}\end{cases},$$ where an $\alpha$-truncated $2$-$g$ basis is one for which each integer in the target set $[\alpha n,(2-\alpha)n]$ can be written as a 2-sum in at least $g$ ways.
\end{thm}

Theorems 3.2 and 3.3 are finite representability versions of the key result in \cite{et}, where a variable input probability was used and the focus was on representing each integer as a sum in {\it logarithmically many} ways; see also \cite{gglz}.

Much of our canonical format for covering threshold progressions can already be seen to be valid; in particular for $h=2$ and $\a=1/2$, Theorem 3.3 reveals that an extra input component of $4\ln\ln n$ yields an extra representation as a sum for each element in $[n/2, 3n/2]$.  Other than improving Theorem 3.3 so as to be valid for all $h$ (which we do not attempt here), all that remains is to address the question of when, wlp/whp, we have the generalized Sidon property of each element in a sumset being represented at most $g\ge 2$ times. 
For $h\ge2; g\ge 1$, we say that $\ca\subseteq[n]$ satisfies the $B_h[g]$ property if for all integers $k\in[h,nh]$, the equation

$$a_1+a_2+\ldots+a_h=k; a_1\le a_2\ldots\le a_h; a_i\in\ca, i=1,\ldots,n$$ has at most $g$ solutions.

\begin{thm}\label{Sidon}
Let $k = k(n)$ be arbitrary. Let $\cc\subseteq [n]$ be a random subset of $[n]$ in which each element of $[n]$ is selected for membership in $\cc$ with probability $p:=\frac{k}{n}$. Then for any $h \geq 2$, $g \geq 1$ we have:

\[
k = o\left ( n^{\frac{g}{h(g+1)}} \right ) \Rightarrow \p(\cc \text{ is } B_h[g]) \to1\quad(n\to\infty),
\]
and
\[ 
n^{\frac{g}{h(g+1)}} = o\left ( k \right ) \Rightarrow \p(\cc \text{ is } B_h[g]) \to0\quad(n\to\infty).\]

\end{thm}

\begin{proof}

 Define
\[\ca_h=\{{\bf a}=(a_1,\ldots, a_h): 1\le a_1\le a_2\le \dots\le a_h\le n\}.\]
We will write $\bf{a}$ in vector form to ensure we have an ordering on the elements, but will also use standard set operations in the obvious way, i.e. ${\bf{a}}\cup {\bf{a'}}=\{a:a\in {\bf{a}}\text{ or }a\in{\bf{a'}}\}$. Next, define

\begin{align*}\cb_{h,g}=\{({\bf a}_1,\ldots,{\bf a}_{g+1})\in \ca_h^{g+1}: a_{1,1}+\ldots a_{1,h}=\ldots=&a_{g+1,1}+\ldots a_{g+1,h}\ \\&{\rm and}\ {\bf a}_1<\ldots<{\bf a}_{g+1}\},\end{align*}
where $<$ denotes the lexicographic order on $\ca_h$.  Finally, set 
\[\cb_{h,g}(l)=\{({\bf a}_1,\ldots,{\bf a}_{g+1})\in\cb_{h,g}:\vert{\bf a}_1\cup\ldots\cup{\bf a}_{g+1}\vert=l\}.\]
Given $\bx \in \cb_{h,g}(l)$, it is convenient to write $\cup \bx:= \{a : \ba \in \bx \textit { and } a \in \ba\}$. For a given $h$ and $g$, the maximum value of $l$ for which $\cb_{h,g}(l)$ is nonempty is $l=h(g+1)$.  Also, notice that $\cc$ satisfies the $B_h[g]$ property if and only if  
it does not contain $\cup \bx$ for any 
$\bx \in\cb_{h,g}$.  Accordingly, for any ${\bf x}
\in\cb_{h,g}$, set 

$$I_{\bf x}=\begin{cases}1&\mbox{if}\ 
\cup\bx\subseteq \cc\\ 0& \mbox{otherwise,}
\end{cases}$$ 
and let 
\[X=\sum_{{\bf x}\in\cb_{h,g}}I_{\bf x}.\]

An element $\bx=(\ba_1,\dots,\ba_{g+1}) \in \mathcal{B}_{h,g}(l)$ is determined by its $l$ distinct elements and a redundancy pattern determining which elements $a_{ij}$ and $a_{i'j'}$ are equal. The number of such redundancy patterns is a constant depending solely on $h$ and $g$. For example, since each of the $\ba_i$ are listed in non-decreasing order such a redundancy pattern could be realized as $g+1$ non-decreasing strings of length $h$ on $\ell$ symbols. Thus, we focus on the number of ways to select the $l$ distinct elements.
%

\begin{prop}\label{Bhg} For $h\geq 2$, $g\geq 1$, and $g+1\leq l\leq h(g+1)$,  $|\mathcal{B}_{h,g}(l)|= O(n^{l-g})$.
\end{prop}
\begin{proof} We have $g$ nondegenerate linear equations to solve, namely 
\begin{align*}
a_{11}+\dots+a_{1h}&=a_{21}+\dots+a_{2h}\\
&\;\;\vdots\\
a_{11}+\dots+a_{1h}&=a_{(g+1)1}+\dots+a_{(g+1)h}.
\end{align*}  

If $l \leq g$, then there are more equations than symbols and the system of equations is thus determined, establishing Proposition \ref{Bhg} in this case. Assume $g+1 \leq l \leq h(g+1)$. Let $S:= \{{\bf a_i}\}_{i=1}^{g+1}$. To prove Proposition $\ref{Bhg}$, we now provide a partition $\{P_1, P_2, \ldots, P_t\}$ of $S$ such that the equations above are determined by  fixing the symbols in precisely one representative of each $P_i$. Intuitively, the procedure is as follows:

Take an arbitrary member of $S$, say $s^1_1$, and initialize $P_1$ with it. Then, if there are any systems of equations given by sums in  $S$   which are determined when we fix the symbols of $s^1_1$, we move one such system, say $S_2$, to $P_1$. Now, if there are any similar systems of equations in $S$ remaining which are determined when we fix the symbols of $s^1_1$ {\em and} $S_2$, we move one, say $S_3$ to $P_1$. Otherwise, we build $P_2$ in the same fashion until $S$ is empty.

Formally, we use the following procedure:

\begin{itemize}
\item[Step 0:] Initialize $i = 1$, $S^{(1)} := S$, and $P_i = \emptyset$.  
\item[Step 1:] If $S^{(i)} = \emptyset$, stop. Otherwise, choose ${\bf a_j} \in S^{(i)}$ and set $P_i := \{{\bf a_j}\}$, $s^i_1:={\bf a_j},$  and remove ${\bf a_j}$ from $S^{(i)}$. 
\item[Step 2:] We have chosen $P_i = \{s^i_1, \ldots, s^i_k\}$. There are two cases:
\begin{itemize}
\item[Case 1:] There is some $S_j \subseteq  S^{(i)}$  such that $S_j = \{\bf{a_j^1},\bf{a_j^2}, \ldots , \bf{a_j^{t_j}}\} $ and the 
system of equations given by $S_j$  is determined by fixing the symbols in $s^1_1, s^2_1, \ldots, s^i_1$ subject to the equations 
\[\sum s^1_1 = \sum s^i_1 =\sum {\bf a_j^1} = \ldots \sum{\bf a_j^{t_j}}.\] 
In this case, set $s^i_{k+1} := {\bf a^1_j}$, $s^i_{k+2} :={\bf a^2_j}, \ldots, s^i_{k+{t_j}} := {\bf a^{t_j}_j} $    and add these elements  to $P_i$. Remove $S_j$ from $S^{(i)}$. Return to Step 2. 
\item[Case 2:] There is no $S_j$ as in Case 1.  Increment $i$ and initialize $P_i:= \emptyset$. Return to Step 1.
\end{itemize}
\end{itemize}

By construction, if we (sequentially) fix the symbols in $s^1_1$, $s^2_1$, $\ldots$, $s^t_1$, then we determine every member of $S$, and so the partition produced by the above procedure has the property we desire.  Let $l_i$ denote the number of free symbols in $s^i_1$ when the symbols in $s^1_1$, $s^2_1$,  $\ldots$, $s^{(i-1)}_1$ have been fixed. Noting that fixing the symbols in $s^1_1$ determines the common sum of each of the $\{\bf a_j\}$, we have 

\[n^{l_1+\sum_{j=2}^r(l_j-1)}=n^{\sum_{j=1}^rl_j-(r-1)}\]
choices for the free symbols. Let $|P_i|= g_i$ and let the number of free symbols  (once the preceeding symbols are fixed) in $s^i_1$ be $l_i$. By construction (critically) the number of sums ($g_j-1$) which can be determined by fixing the symbols in $s^j_1$ is the same as the number of symbols that can be determined by fixing the symbols in $s_1^j$. Thus  
\[\sum_{j=1}^rl_j+(g_j-1)= l ,\] and the number of choices overall is at most
\[n^{l-\sum_{j=1}^r(g_j-1)-(r-1)}=n^{l-(g+1)+r-(r-1)}=n^{l-g},\]
as desired.
\hfill\end{proof}
  Since $\cc$ is $B_h[g]$ if and only if $I_\bx=0$ for all $\bx\in \mathcal{B}_{h,g}$, i.e. if and only if $X=0$, we  apply Markov's inequality to get

$$
\p(\cc\text{ is not }B_h[g])=\p(X\geq 1)\le \e[X].$$ Furthermore for all $\bx\in \mathcal{B}_{h,g}(l)$,
$$
\p(I_\bx=1)= \left(\frac{k}{n}\right)^l,
$$
and thus 
$$
\e[X]= \sum_{l=g+1}^{h(g+1)}|\mathcal{B}_{h,g}(l)|\left(\frac{k}{n}\right)^l\preceq \sum_{l=g+1}^{h(g+1)}n^{l-g}\left(\frac{k}{n}\right)^l\to0
$$
if $k\ll n^{\frac{g}{h(g+1)}},$ proving the first part of Theorem~\ref{Sidon}.  

The proof of the second part begins by setting 
$$
Y=\sum_{\bx\in \mathcal{B}_{h,g}(h(g+1))}I_\bx,
$$
so that
$$
\p(\cc\text{ is }B_{h}[g])=\p(X=0)\leq \p(Y=0).
$$

Define  a relation $\sim$ on $\mathcal{B}_{h,g}(h(g+1))$ as follows: For $\bx
$,  $\by
\in \mathcal{B}_{h,g}(h(g+1))$, we have 
$$
\bx\sim\by\iff \bx\neq\by \text{ and } 
(\cup \bx) \cap (\cup \by )\neq \emptyset.
$$
Applying Janson's Inequality (Theorem~8.1.1 in~\cite{as}) we see:
\beq \label{JISidon}
\p(Y=0)\leq \left(\prod_{\bx\in \mathcal{B}_{h,g}(h(g+1))}\p(I_\bx=0)\right)\exp(\Delta),
\eeq
with
\beq \label{JIDSidon}
\Delta=\sum_{\bx\sim\by}\p(I_\bx I_\by=1).
\eeq

With a view towards bounding $\Delta$, 
for $(g+1)h\leq l\leq 2(g+1)h-1$, define 
$$
\mathcal{D}_{h,g}(l):=\{(\bx,\by)\in \mathcal{B}_{h,g}(h(g+1))\times\mathcal{B}_{h,g}(h(g+1)):\bx\sim\by\text{ and }|\bx\cup\by|=l\},
$$
so that 
\beq \label{JID2Sidon}
\Delta=\sum_{l=h(g+1)}^{2h(g+1)-1}|\mathcal{D}_{h,g}(l)|p^l.
\eeq

\begin{lem}\label{Dhg}
For $h \geq 2$, $g \geq 1$ and $h(g+1) \leq l \leq 2h(g+1) -1$, we have $|\mathcal{D}_{h,g}(l)|p^l = o(n^{h(g+1) -g}p^{h(g+1)})$. 
\end{lem}

\begin{proof}
Our aim is to approximate the number of pairs $(\bx, \by)\in \mathcal{D}_{h,g}(l)$. Given $(\bx, \by) \in \mathcal{D}_{h,g}(l)$ with $\bx = (\ba_1, \ba_2, \ldots, \ba_{g+1})$ and $\by = ({\bf b}_1, {\bf b}_2, \ldots, {\bf b}_{g+1})$ (with each entry indexed in the natural way), let $r = r(\bx, \by)$ denote the number of indices $i$ so that ${\bf b}_i \subseteq \cup \bx$. First, we remark that by Proposition~\ref{Bhg}, we have $O(n^{h(g+1) -g})$ choices for $\bx$, since $\bx$ contains $h(g+1)$ distinct symbols drawn from $[n]$ subject to the same $g$ linear restrictions as before.
%

While $\bx \in B_{h,g}(h(g+1))$ has all distinct symbols (and therefore no redundancy pattern), $(\cup \bx) \cap (\cup \by) \neq \emptyset$ and so we will use a redundancy pattern for $\by$ to tell us which variables it shares with $\bx$. However, there are only constantly many redundancy patterns for $\by$ given the choices for $\bx$. We thus focus on how many choices we have for the $\ell - h(g+1)$ elements of $\cup \by$ disjoint from $\cup \bx$. We split into three cases, when $1 \leq r \leq g$, and  the extreme cases  $r = 0$ and $r = g+1$.    

\vskip .5cm
 \noindent \textit{Case 1: $1 \leq r \leq g$}:

$1 \leq r \leq g$ means that there are  $r$ indices $j$ so that ${\bf b}_j \subseteq \cup \bx$, and that this is not the totality of $[g+1]$. 
Suppose without loss of generality that $1$ is one such index and throw away all $r$ of these indices so that  $\{i_j\}_{j=1}^{g+1-r}$ is the collection of indices for which ${\bf b}_{i_j} \not \subseteq \cup \bx$ for each $j \in [g+1-r]$. Noting that since $\by \in \mathcal{B}_{h,g} (h(g+1))$ means precisely that $\by$ has all distinct symbols, the $g+1-r$ linear equations given by:  

\begin{align*}
b_{11}+\dots+b_{1h}&=b_{{i_1} 1}+\dots+b_{{i_1},h}\\
&\;\;\vdots\\
b_{11}+\dots+b_{1h}&=b_{i_{g+1-r} 1}+\dots+b_{i_{g+1-r}h}
\end{align*}

\noindent each have a variable not contained in any other equation, and are hence non-degenerate. Hence by Proposition \ref{Bhg} there are at most $O(n^{l - h(g+1) -(g+1-r)})$ ways to pick $\by$ in this case, and thus $O(n^{l + r -2g -1})$ pairs $(\bx, \by)$. Thus we have established Lemma~\ref{Dhg}, Case 1 if we show $n^{l +r -2g -1}p^l \ll n^{h(g+1)-g}p^{h(g+1)}$. Rearranging, we need

\begin{equation}\label{npeq}
(np)^{l-h(g+1)} \ll n^{g+1-r}.
 \end{equation}
 We remark that since $np = k$, $r \le g$, and we are assuming $k \gg n^{\frac{g}{h(g+1)}}$,  we can write $np = \phi(n) n^{\frac{g}{h(g+1)}}$ for some $\phi(n) \rightarrow \infty$. We will show that we can produce $\phi^\prime(n) \rightarrow \infty$ so that 

\begin{equation} \label{phiprime}
(\phi^\prime(n) n^{\frac{g}{h(g+1)}})^{l-h(g+1)} \ll n^{g+1-r}
\end{equation}

\noindent holds, and that  Equation~\ref{phiprime} is sufficient to imply Lemma~\ref{Dhg}, Case 1. We can find a $\phi^\prime(n)$ which satisfies Equation~\ref{phiprime} whenever $n^{\frac{g(l-h(g+1))}{h(g+1)}} \ll n^{g+1-r}$. In other words, if 

\[
\frac{g(l-h(g+1))}{h(g+1)} < g+1-r.
\] 

Since $l = |\cup \bx \bigcup \cup \by|$, and $r$ is the number of indices $j$ for which all $h$ of the symbols of ${\bf b}_j$ occur in $\cup \bx$, we get the elementary bound $l \leq 2h(g+1) -rh$. Further, $r \leq g$ and the above inequality follows readily. Hence, we can find $\phi^\prime(n)$ so that Equation~\ref{phiprime} holds. To see that this implies Lemma~\ref{Dhg}, Case~1 note that if $\phi(n) = O(\phi^\prime(n))$ then Equation~\ref{npeq} follows from Equation~\ref{phiprime} by simple substitution. If $\phi(n) \ll \phi^\prime(n)$, then we have  
\[
(np)^{l-h(g+1)}= (\phi(n) n^{\frac{g}{h(g+1)}})^{l-h(g+1)}  \ll (\phi^\prime(n) n^{\frac{g}{h(g+1)}})^{l-h(g+1)} \ll n^{g+1-r}.
\]

\noindent by Equation~\ref{phiprime}, as desired. On the other hand, if $\phi(n) \gg \phi^\prime(n)$, we have:
\[
p = \phi(n) n^{\frac{g}{h(g+1)}-1} \ge \phi^\prime(n) n^{\frac{g}{h(g+1)}-1} = p^\prime
\]

\noindent and the property  ``$\cc$ is $B_h[g]$'' is monotone in $p$. This concludes the proof of Lemma~\ref{Dhg}, Case 1.

\vskip .5cm

\noindent \textit{Case 2: $r = 0$}:

$r = 0$ means that ${\bf b}_i \not \subseteq  \cup \bx$ for all $1 \leq i \leq g+1$. Since $\by \in \mathcal{B}_{h,g} (h(g+1))$,  $\by$ has all distinct symbols and the $g$ linear equations given by:  

\begin{align*}
b_{11}+\dots+b_{1h}&=b_{21}+\dots+b_{2h}\\
&\;\;\vdots\\
b_{11}+\dots+b_{1h}&=b_{(g+1)1}+\dots+b_{(g+1)h}
\end{align*}
each have a variable not contained in any other equation, and are hence non-degenerate. Thus,  there are $O(n^{l -h(g+1) -g})$ ways to pick $\by$ and  $O(n^{l - 2g})$ pairs $(\bx, \by)$ for which $r = 0$.  For $1 \leq r \leq g$, $l +r - 2g - 1 \geq l - 2g$, and so we are done by Case 1.

\vskip .5cm
\noindent \textit{Case 3: $r = g+1$}:

$r=g+1$ means that for each index $i$, ${\bf b}_i \subseteq \cup \bx$, i.e., $\cup \bx = \cup \by$ and $l = h(g+1)$. Notice that, a priori,  choosing $\bx$  arbitrarily in  $n^{h(g+1) - g}$ ways subject to the same linear constraints as in Proposition~\ref{Bhg} completely determines $\by$, and $|D_{h,g}(h(g+1))| p^{h(g+1)} = O(n^{h(g+1) - g}) p^{h(g+1)}$, contrary to the conclusion of Lemma~\ref{Dhg}. However, not every choice of $\bx$ satisfies the linear constraints imposed by $\by$. We want to show that these constraints are non-trivial. More precisely, we wish to show that we can find $i$, $j$ so that the $g+1$  linear equations given by:

\begin{align*}
a_{11}+\dots+a_{1h}&=a_{21}+\dots+a_{2h}\\
&\;\;\vdots\\
a_{11}+\dots+a_{1h}&=a_{(g+1)1}+\dots+a_{(g+1)h}\\
b_{i1} + \ldots + b_{ih} &= b_{j1} + \ldots + b_{jh}
\end{align*}

\noindent are non-degenerate. Then we will have at most $O(n^{h(g+1) - g-1})$ choices for $(\bx, \by)$ and thus
 \[|D_{h,g}(h(g+1))| p^{h(g+1)} = O(n^{h(g+1) - g-1}) p^{h(g+1)} = o(n^{h(g+1) - g}) p^{h(g+1)}\]
\noindent as desired. We now produce such an $i$ and $j$. 

First,  there are at least two indices (say $1$ and $2$) so that neither  ${\bf b}_1$ or ${\bf b}_2$ is any member of $\bx$, for if there is at most one such index the lexicographic ordering on $\bx$ and $\by$ implies that $\bx = \by$. To see that the equations
\begin{align*}
a_{11}+\dots+a_{1h}&=a_{21}+\dots+a_{2h}\\
&\;\;\vdots\\
a_{11}+\dots+a_{1h}&=a_{(g+1)1}+\dots+a_{(g+1)h}\\
b_{11} + \ldots + b_{1h} &= b_{21} + \ldots + b_{2h}
\end{align*}
\noindent 
are non-degenerate, we argue formally as follows:  Suppose we have a linear combination
\[ 
\sum_{i = 1}^g c_i(a_{11}+\dots+a_{1h} - (a_{i1}+\dots+a_{ih})) =   b_{11} + \ldots + b_{1h} - (b_{21} + \ldots + b_{2h}).
\]
For each $i$, $c_i \in \{-1,1,0\}$ as the symbols in $\bx$ are all distinct (and hence show up in at most one equation) while the coefficients on the right are unitary. Moreover, there is exactly one $j$ for which $c_j \neq 0$, as otherwise we have more symbols on the left hand side than the right. Finally, the equation 
\[\pm (a_{11}+\dots+a_{1h} - (a_{j1}+\dots+a_{jh})) = b_{11} + \ldots + b_{1h} - (b_{21} + \ldots + b_{2h})
\]
\noindent forces ${\bf b}_1$, ${\bf b}_2 \in \{\ba_1, \ba_j\}$ by matching coefficients. This contradicts our assertion about ${\bf b}_1$ and  ${\bf b}_2$, and so we have found the indices we desired. Thus, Lemma~\ref{Dhg}, Case 3 is finished.  
\end{proof}

Lemma~\ref{Dhg}  along with Equation~\ref{JID2Sidon}, now yield:

\[
\Delta=\sum_{l=h(g+1)}^{2h(g+1)-1}|\mathcal{D}_{h,g}(l)|p^l = o(n^{h(g+1) - g} p^{h(g+1)}).
\] 
\noindent Hence this estimate for  $\Delta$, together with Proposition~\ref{Bhg} (for $l = h(g+1)$) yields, when substituted into Equation~\ref{JISidon}:
\begin{align*}
\p(Y=0)&\leq \left(\prod_{\bx\in \mathcal{B}_{h,g}(h(g+1))}\p(I_\bx=0)\right)\exp(\Delta)\\
& \leq (1- p^{h(g+1)})^{|\mathcal{B}_{h,g}(h(g+1))|} \exp(\Delta)\\
& \preceq \exp(-n^{h(g+1) - g}p^{h(g+1)})\exp(\Delta)\\
& \rightarrow 0,
\end{align*} 
as desired. 
\end{proof}
\section {Permutations} 
What is the minimum number $C_1(n,n+1)$ of $(n+1)$-permutations needed to cover each $n$-permutation as an embedded order-isomorphic subsequence?  We have from \cite{aghk} that a bound is
$$C_1(n,n+1)\leq\frac{(n+1)!}{n^2}\lr 1+{\log n}\rr(1+o(1)).$$
Also, it was shown in the same paper that
$$C_\lambda(n,n+1)\leq\frac{(n+1)!}{n^2}\lr \frac{\lambda}{(\lambda-1)!}(1+o(1))+{\log n}+(\lambda-1)\log\log n\rr,$$
once again exhibiting the $\log\log$ phenomenon in the context of bounds.
Denote by $S_n$ the collection of permutations on $n$ symbols. The authors of \cite{aghk} provide the threshold  for the property that each $\pi\in S_n$ is covered by at least one permutation in $S_{n+1}$. This statement is made precise in Theorem~\ref{TCPerm}:

\begin{thm}(\cite{aghk})\label{TCPerm}
Let $r = r(n)$ be arbitrary.  Let $\cc\subseteq S_{n+1}$ be a random subset of $S_{n+1}$ in which each element of $S_{n+1}$ is selected for membership in $\cc$ with probability $p=\frac{\lr\log n-1+\frac{\log n}{2n}+ \frac{r}{n}\rr}{n}$. Let $X$ denote the number of elements of $S_n$ which are  not contained as order-isomorphic patterns of least one member  of $\cc$.  We have:
 
\[
r \rightarrow + \infty  \Rightarrow \p(X = 0) \to 1,
\]
and
\[ 
r \rightarrow - \infty\Rightarrow \p(X = 0) \to 0.\]

\end{thm} 

In the main new results of this section, Theorem~\ref{TLCPerm} extends Theorem~\ref{TCPerm} to  $\lambda$-coverings, and Theorem~\ref{TPPerm} provides the complementary $\lambda$-packing result.  First, however, we state Lemmas~\ref{nsqr},~\ref{ncube}, and ~\ref{four} from \cite{aghk}, as they will each be useful for us.

\begin{lem}(\cite{aghk})\label{nsqr}
Let $c(n,\pi)$ denote the number of permutations in $S_{n+1}$ which cover a fixed $\pi \in S_n$. We have $c(n,\pi) = c(n,\pi^\prime) = n^2+1$ for all $\pi, \pi^\prime \in S_{n+1}$. 
\end{lem}  

\begin{lem}(\cite{aghk})\label{ncube}
For any $\pi \in S_n$, the set:

\[
\mathcal{J}_\pi:= \{\pi^\prime \in S_n: \pi \text{ and } \pi^\prime \text{ can be jointly covered by some } \rho\in S_{n+1}\}
\]

has cardinality at most $n^3$. 
\end{lem}
 
 \begin{lem}(\cite{aghk})\label{four}
 For any $\pi, \pi^\prime \in S_{n}$, the set:
 
 \[
 C_{\pi, \pi^\prime} := \{ \rho \in S_{n+1}: \rho \text{ covers } \pi \text{ and } \pi^\prime \text{ jointly}\}.
 \]
 has cardinality at most $4$. 
 \end{lem}
 
 We are now ready to state Theorem~\ref{TLCPerm}:
 
\begin{thm}\label{TLCPerm}

Let $r = r(n)$ be arbitrary and let $\lambda \geq 1$.  Let $\cc\subseteq S_{n+1}$ be a random subset of $S_{n+1}$ in which each element of $S_{n+1}$ is selected for membership in $\cc$ with probability $$p=\frac{1}{n^2} \lc n\ln n - n +(\lambda-1) \ln n + (\lambda-1) \ln \ln n - \ln(\lambda-1)!+\frac{\ln n}{2}+r \rc.$$ Let $X=X_\lambda$ denote the number of elements of $S_n$ which are  not covered  by at least $\lambda$ members  of $\cc$.  We have:
 
\[
r \rightarrow + \infty  \Rightarrow \p(X = 0) \to 1,
\]
and
\[ 
r \rightarrow - \infty\Rightarrow \p(X = 0) \to 0.\]

\end{thm}

\begin{proof}
Lemma~\ref{nsqr} states that each member of $S_n$ is covered by precisely $n^2+1$ members of $S_{n+1}$. Moreover, $n!p \rightarrow \infty$. Hence,  by Markov's inequality, Lemma~\ref{tails}, and Lemma~\ref{nsqr} we have:

\begin{align*}
\p(X \geq 1) &\leq \e(X)\\
& = n! \sum_{j =0}^{\lambda-1}\binom{n^2+1}{j} p^j (1-p)^{n^2+1-j}\\
& = n!\binom{n^2+1}{\lambda-1}p^{\lambda-1}{(1-p)}^{n^2 -\lambda +2}(1+o(1))\\
& = \sqrt{2 \pi n} \left(\frac{n}{e}\right)^n \frac{(n^2+1)^{\lambda-1}}{(\lambda-1)!}p^{\lambda-1} \exp\lc -pn^2 (1+o(1)) \rc (1+o(1))\\
&\to0\quad(r\to\infty), \numberthis \label{cpermex}
\end{align*}
proving the first part of the result. For the second part of the theorem, we employ the Stein-Chen method, noting first that $\e(X)\to\infty$ with $p$ as above and $r\to-\infty$.   Following the process in   \cite{aghk}, we begin by setting $$X=\sum_{j=1}^{n!}I_j,$$
where for each $j \in [n!]$, we set  $I_j=1$  if the $\pi_j$  is covered $\lambda-1$ or fewer times ($I_j=0$ otherwise). As before, for each $j\in [n!]$, we seek a coupling $\{J_{ji}\}_{1\le i\le{n!}}$ that satisfies:

\[\cl\lr J_{j1},\ldots, J_{j{n!}}\rr=\cl\lr I_1,\ldots, I_{n!}\vert I_j=1\rr,\]  

We (partially) exhibit this coupling as follows:  If $I_j=1$, i.e., if $\pi_j$ is covered by at most $\lambda-1$ $(n+1)$-permutations, we let $J_{ji}=I_i$ for each $i$. On the other hand, if the $\pi_j$ is covered $\lambda$ or more times, we deselect a certain number of $(n+1)$-permutations (according to the appropriate distribution) so as to achieve a sample outcome corresponding to $I_j=1$.  We then set $J_{ji}=1$ if the $\pi_i$  is covered $\lambda-1$ or fewer times after this is done.  As before,  the conditional distribution is attained by a process of deselection, so we  must have $J_{ji}\ge I_i$ for each $i\ne j$ so that the indicators are positively related. Via Lemma~\ref{SC}, we see that:

\begin{eqnarray}
\tv(\cl(X), \P(\mu))
&\leq&\frac{1}{\mu}\left(\v(X) -\mu +2\sum_{i=1}^{n!} \p^2(I_i =1)\right)\label{tv}
\end{eqnarray}
For $i,j \in [n!]$ write $i\sim j$ whenever $I_i$ and $I_j$ are not independent events, so:
\[
\v(X) = \sum_j \lr\e(I_j) - \e^2(I_j)\rr + \sum_{i \sim j} \lr\e(I_iI_j) - \e(I_i) \e(I_j)\rr.
\]
By Lemma~\ref{ncube}, $I_i$ is not independent of only the members of $\mathcal{J}_{I_i}$, which has cardinality at most $n^3$. Moreover, by Lemma~\ref{nsqr}, $\pi_i$ is covered by precisely $n^2+1$ $(n+1)$-permutations.  Hence, we have:

\begin{align*}
\frac{\v(X)}{\mu} - 1 &\leq \frac{\sum_{i \sim j} (\e(I_iI_j) - \e(I_i) \e(I_j))}{\mu}\\
& \leq n^3\left(\frac{\max_{i \sim j}\p(I_iI_j =1)- \binom{n^2+1}{\lambda-1}^2p^{2\lambda-2} (1-p)^{2n^2 +4 -2\lambda}}{\binom{n^2+1}{\lambda-1}p^{\lambda-1}(1-p)^{n^2+2-\lambda}(1+o(1))}\right)
\end{align*}
Plugging into inequality \eqref{tv}, we see 
\begin{eqnarray}
\tv(\cl(X), \P(\mu)) &\leq& {n^3}\left(\frac{\max_{i \sim j}\p(I_iI_j =1)- \binom{n^2+1}{\lambda-1}^2p^{2\lambda-2} (1-p)^{2n^2 +4 -2\lambda}}{\binom{n^2+1}{\lambda-1}p^{\lambda-1}(1-p)^{n^2+2-\lambda}(1+o(1))}\right)\nonumber\\{}&&+2\frac{\sum_{i=1}^{n!} \p^2(I_i =1)}{\mu} \label{tv2}
\end{eqnarray}
Consider the second term.  By Lemma 2.2,
\begin{align*}
2\frac{\sum_{i=1}^{n!} \p^2(I_i =1)}{\mu} & =
2 \p(I_1 =1)\\
&=\binom{n^2+1}{\lambda-1}p^{\lambda-1}(1-p)^{n^2+2-\lambda}(1+o(1))\\
&\le\frac{n^{2\lambda-2}}{(\lambda-1)!}p^{\lambda-1}e^{-n^2p}(1+ o(1)) \\
&=o(1),
\end{align*}
and we now seek to show that the first term in (\ref{tv2}) is $o(1)$ as well. Fix $i,j \in [n!]$. Since by Lemma \ref{four}, $\pi_{i}$ and $\pi_{j}$ are co-coverable by at most $4$ members of $S_{n+1}$, we let $a_{i,j}\le4$ denote the number of permutations which co-cover $\pi_i$ and $\pi_j$ and denote by $A_{i, j}$the number of these $a_{i,j}$ permutations which are selected. 
Then by Lemmas \ref{nsqr} and \ref{tails} we have that:
\begin{align*}
\p(I_i I_j = 1) =&\sum_{t=0}^{a_{i,j}}\p(A_{i,j}=t)\p(I_iI_j=1\vert A_{i,j}=t)\\
&\le \sum_{t=0}^{a_{i,j}}\p(I_iI_j=1\vert A_{i,j}=t)\\
& \leq 
 \lr\sum_{r = 0}^{\lambda-1-t}  \binom{n^2 +1-t}{r}p^r(1-p)^{n^2+1-t-r}\rr^2 \\
&\le  \lr\sum_{r = 0}^{\lambda-1}  \binom{n^2 +1-t}{r}p^r(1-p)^{n^2+1-t-r}\rr^2 \\
& =\rho^2(1+o(1)),
\end{align*}
where we denote the probability of $\lambda-1$ successes in $n^2+1$ Bernoulli trials by $\rho$. Plugging into (\ref{tv2}), we have:

\begin{eqnarray}
\tv(\cl(X), \P(\mu)) & \leq& n^3 \left(\frac{(1+o(1))\rho^2 - \rho^2}{\rho}\right) +o(1)\nonumber\\
& =&o(1)n^3\rho\nonumber\\
&\to&0
\end{eqnarray}  
since
\[\rho=O\lr n^{2\l-2}p^{\l-1}e^{-n^2p}\rr,\]which tends to zero if $p=O(\log n/n)$.


Thus $\tv(\cl(X), \P(\mu)) = o(1)$, and in particular there exists $\epsilon_n = o(1)$ such that:
\[
e^{-\mu} - \epsilon_n \leq P(X =0) \leq e^{-\mu} + \epsilon_n.
\]
We know as noted above that $e^{-\mu} \rightarrow 0$ when $r \rightarrow -\infty$, so $P(X = 0) \rightarrow 0$, as desired.
\hfill\end{proof}

Theorem~\ref{TPPerm} establishes a $\lambda$-packings threshold for permutations.

\begin{thm}\label{TPPerm}

 Let $\cc\subseteq S_{n+1}$ be a random subset of $S_{n+1}$ in which each element of $S_{n+1}$ is selected for membership in $\cc$ with probability $p$. Let $X=X_\lambda$ denote the number of elements of $S_n$ which are  not covered  by at most $\lambda$ members  of $\cc$.  We have:
 
\[
p\ll\frac{1}{n^2}\frac{1}{n!^{1/(\l+1)}} \Rightarrow \p(X = 0 )\to 1,
\]
and
\[ 
p\gg\frac{1}{n^{2\l/(\l+1)}}\frac{1}{n!^{1/\l+1}}\Rightarrow\p(X=0)\to0.\]

\end{thm}

\begin{proof}
As before, let $S_n = \{\pi_i\}_{i=1}^{n!}$, and let $X = \sum_{i=1}^{n!} I_i$, where $I_i$ is the indicator function for the event that $\pi_i$ is covered by at least $\lambda+1$ members of $S_{n+1}$.  We establish the first half of the theorem by Markov's inequality and Lemma 4.2:

\begin{eqnarray*}\p(X\ge 1)\le\e(X)&=&n!\sum_{j=\l+1}^{n^2+1}{{n^2+1}\choose{j}}p^j(1-p)^{n^2+1-j}\\
&\le&n!{{n^2+1}\choose{\l+1}}p^{\l+1}(1-p)^{n^2-\l}(1+o(1))\\
&\le&n!\frac{n^{2\l+2}}{(\l+1)!}p^{\l+1}(1+o(1))\\
&\to&0
\end{eqnarray*}
whenever $p \ll\frac{1}{n^2}\frac{1}{n!^{1/(\l+1)}}$. To establish the second part of the theorem, we seek to employ Talagrand's inequality. $X$ is $(n+1)$-Lipschitz, as reflipping the coin to determine membership of any $(n+1)$-permutation in $\cc$ can affect the value of $X$ by at most $(n+1)$, as each member of $S_{n+1}$ covers at most $n+1$ members of $S_n$. Moreover, $X$ is $s(\lambda+1)$-certifiable as the event $\{X \ge s\}$ can be certified by the outcomes of $s(\lambda+1)$ trials. 
We apply Talagrand's Inequality as in Theorem~\ref{Tballs} 
to see: 

\[\p(X=0)\le 2\exp\{-{\rm Med}(X)/(4(n+1)^2(\l+1))\}\to0\]
provided that ${\rm Med}(X)\gg n^2$. Since the median and mean of $X$ differ by at most $40(n+1)\sqrt{(\lambda+1)\e(X)}\le120n\sqrt{\l\e(X)}$ as per Fact 10.1 in \cite{mr}, we have that \[\left\vert\frac{{\rm Med}(X)}{\e(X)}-1\right\vert\le\frac{120n\sqrt{\l}}{\sqrt{\e(X)}}\to0\]
if $\e(X)\gg n^2$.  Thus ${\rm Med}(X)/\e(X)\to1$ if $\e(X)\gg n^2$ and hence $\p(X=0)\to 0$ if $\e(X)\gg n^2$.  But 
\begin{eqnarray*}\e(X)&=&n!\sum_{j=\l+1}^{n^2+1}{{n^2+1}\choose{j}}p^j(1-p)^{n^2+1-j}\\
&\ge&n!{{n^2+1}\choose{\l+1}}p^{\l+1}(1-p)^{n^2-\l}\\
&\ge&n!\frac{n^{2\l+2}}{(\l+1)!}p^{\l+1}(1+o(1))\\
& \gg& n^2
\end{eqnarray*}
if $p\gg\frac{1}{n^{2\l/(\l+1)}}\frac{1}{n!^{1/\l+1}}$, as desired. We remark  that the slight gap in the upper and lower threshold in Theorem~\ref{TPPerm}  is an artifact of Talagrand's inequality.  That the true single threshold is at $\frac{1}{n^2}\frac{1}{n!^{1/(\l+1)}}$ might be provable via a pedestrian technique such as the second moment method.
 \end{proof}

\section{Union-Free Families}

Let $\mathcal{P}([n])$ denote the power set of $[n]$. We say that a family $\cc \subseteq \mathcal{P}([n])$ is {\em weakly union-free} if there are no 4 distinct elements $\{A,B,C,D\} \subseteq \cc$ such that $A \cup B = C \cup D$ (see \cite{ff}). Let $C(n)$ is the maximum size of such a family. Frankl and F{\"u}redi \cite{ff} have used probabilistic methods and information theory to introduce the following bounds:
\[2^{(n-\log 3)/3} -2 \leq C(n)\leq 2^{(3n+2)/4} \sim 2^{1/2}\cdot 1.68^n\]
Since then the upper bound has been improved by Coppersmith and Shearer \cite{c} to
\[C(n)\leq 2^{[0.5 +o(1)]n}.\]

In this section, we present a packing threshold version of this question in Theorem~\ref{WPT} below,
in which we are ``packing'' a family of sets until whp or wlp each {\em union} is ``covered'' at most once. 

\begin{thm}\label{WPT}
Let $\cc\subseteq \mathcal{P}([n])$ be a random subset of $\mathcal{P}([n])$ in which each element of $\mathcal{P}([n])$ is selected for membership in $\cc$ with probability $p $. Let $X$ denote the number of distinct quadruples   $\{A,B,C,D\}\subseteq \cc$ such that $A \cup B = C \cup D$. Notice that $X = 0$ if and only if $\cc$ is weakly union-free.  We have:
 
\[
p \ll \left(\frac{1}{10}\right)^{n/4} \Rightarrow \p(X = 0) \rightarrow 1,
\]
and
\[ 
p \gg \left(\frac{1}{10}\right)^{n/4} \Rightarrow \p(X = 0) \to 0.\]
\end{thm}

\begin{proof}
Let $\mathcal{B}: = \{\{A,B,C,D\} \text{ distinct }: A \cup B = C \cup D \}$. Clearly, $X = 0$ (e.g., $\cc$ is  weakly union-free) if and only if no member of $\mathcal{B}$ is a subset of $\cc$. We first count the size of $\mathcal{B}$. Let $U$ be some $k$-set ($1 \leq k \leq n$).  Given  $f:U \rightarrow \{0,1,2\}$, we say $f$ {\em uniquely determines} the ordered pair of sets $(R, S)$  if the following holds:

\[
\text{ for } x \in U  \begin{cases} & x \in R \setminus S \text{ if } f(x) = 0 \\ & x \in S \setminus R \text{ if } f(x) = 1\\& x \in R \cap S \text{ if } f(x) = 2 \end{cases}  
\]

Notice that each map $f:U \rightarrow \{0,1,2\}$  uniquely determines some pair  $(R,S)$. We say $f$ {\em determines} the unordered pair of sets $\{R,S\}$ if $f$ uniquely determines either of $(R,S)$ or $(S,R)$.  For $4 \leq k \leq n$ there are $\binom{n}{k}$ ways to pick $U$, which we will  write as a union in two different ways. There are $\frac{3^k -3}{2}$ non-constant determining maps which determine distinct sets, 
and so there are $\binom{\frac{3^k -3}{2}}{2}$ ways to determine  $A$, $B$, $C$, and $D$ so that $A \cup B = C \cup D$. Hence we have:

\begin{align*}
|\mathcal{B}| &= \sum_{k=3}^{n} \binom{n}{k}\binom{\frac{3^k -3}{2}}{2}\\
& =  \frac{1}{8} \sum_{k=3}^{n} \binom{n}{k} 9^k (1+ o(1))\\
& = \frac{1}{8} 10^n (1+o(1)).\numberthis \label{Badas}
\end{align*}

We now prove the easy first half of the theorem using Markov's inequality. For any $\{A,B,C,D\}\in \mathcal{B}$ define $I_{A,B,C,D}$ to be the indicator random variable which is $1$ if $\{A,B,C,D\} \subseteq \cc$ and $0$ otherwise so that:

\[
X = \sum_{\{A,B,C,D\} \in \mathcal{B}} I_{A,B,C,D}
\]

By Markov's inequality, we have:

\begin{align*}
\p(X \geq 1) &\leq \e(X)\\
& = \sum_{\{A,B,C,D\} \in \mathcal{B}} \p(I_{A,B,C,D}=1)\\
&= \frac{1}{8} p^4 10^n (1+o(1)) 
\end{align*}
Hence, if $p \ll \left(\frac{1}{10}\right)^{n/4}$, we have $\p(X = 0) \rightarrow 1$, as desired. 

 To see the second half, we turn to Janson's inequality as in the proof of Theorem~\ref{Sidon}.  To this end, we define a relation $\sim$ on $\mathcal{B}$ such that for $R, S \in \mathcal{B}$, we have $R \sim S$ if and only if $R \neq S$, and $R \cap S \neq \emptyset$. Janson's inequality yields:
 
 \[
 \p(X = 0) \leq \left( \Pi_{\{A, B, C, D\}\in \mathcal{B}} \p(I_{A,B,C,D}=0)\right)\exp(\Delta)\numberthis \label{JanWuf1}
 \]
 
 where, $\Delta$ is given by:
 
 \[
 \Delta:= \sum_{\{A,B,C,D\}\sim\{E,F,G,H\}}\p(I_{A,B,C,D} I_{E,F,G,G} =1)
 \]
 
 For each $\ell$, we partition $\Delta$ into classes, given by:
 
 \[
 \mathcal{D}(\ell) := \{(R,S)\in \mathcal{B}^2: |R\cup S|=\ell\}
 \]

 Note that taking $5 \leq \ell \leq 7$ partitions all of $\Delta$. We have the following claims about the size of each of the $\mathcal{D}(\ell)$. 
 
 \begin{claim}\label{del5}
 \begin{align*}
 |\mathcal{D}(5)| &\leq 2\sum_{i=0}^n \binom{n}{i}\left(2^i\right)^2\sum_{j=0}^{n-i}\binom{n-i}{j}3^{i+j}\\
  &= O(16^n).
 \end{align*}
 \end{claim}

\begin{claim}\label{del6}
 \begin{align*}
 |\mathcal{D}(6)|& \leq 2\sum_{i=0}^n \binom{n}{i} 2^i \sum_{j=0}^{n-i} \left(3^{i+j}\right)^2\\
 & =O(28^n). 
 \end{align*}

\end{claim}

\begin{claim}\label{del7}

\begin{align*}
|\mathcal{D}(7)|& \leq \sum_{i=0}^n \binom{n}{i}\left(2^i\right)^2 \sum_{j=0}^{n-i} \binom{n-i}{j} 3^{i+j} \sum_{k=0}^{n-i} \binom{n-i}{k} 3^{i+k}\\
& = O(52^n). 
\end{align*}
\end{claim}

\begin{proof}[Proof of Claim~\ref{del5}]

By definition,   $(\{A,B,C,D\} ,\{E,F,G,H\}) \in \mathcal{D}(5)$ means that $|\{A,B,C,D\} \cup\{E,F,G,H\}| = 5$, (i.e., there are only 5 distinct sets present). Hence, up to relabeling, we can call these sets $\{A,B,C,D\}$ and $\{A,B,C,E\}$. Further, each of $\{A,B,C,D\}\in\mathcal{B}$ and $\{E,F,G,H\}\in\mathcal{B}$, and so they each serve as obstacles to the weakly union-free condition. We assume without loss of generality that $A \cup B = C \cup D$ always. Up to relabeling, we can partition $\mathcal{D}(5)$ into two families:

\[
\mathcal{D}_1(5) := \{(\{A,B,C,D\},\{A,B,C,E\}): A \cup B = C \cup D; A \cup B = C \cup E\}
\]  

or

\[
\mathcal{D}_2(5) := \{(\{A,B,C,D\},\{A,B,C,E\}):A \cup B = C \cup D; A \cup C = B \cup E\}
\]

We will provide an upper bound on each $\mathcal{D}_1(5)$ and $\mathcal{D}_2(5)$. To count the number of members of $\mathcal{D}_1(5)$,  suppose first that  $|C| = i$. There are $\binom{n}{i}$ ways to pick $C$. We  pick two subsets of $C$ (which will serve as $C \cap D$ and $C \cap E$) in at most $\left(2^i\right)^2$ ways. Suppose that $C \cup D$ has size $i + j$. We pick the remaining elements of $D$ in $\binom{n-i}{j}$ ways. We have now completely determined the sets $C$ and $D$, and thus $C\cup D$. The first equation reveals that we have also determined $A \cup B$. There are at most $3^{i+j}$ ways to pick a determining map to pick sets $A$ and $B$. Finally, since we know $A \cup B$, the second equation reveals that we know $C \cup E$. We have also determined $C$ and $C \cap E$, and so we have determined $E$.  Summing over $i$ and $j$ we have:

\begin{align*}
|\mathcal{D}_1(5)| &\leq \sum_{i=0}^n \binom{n}{i} \left(2^i\right)^2 \sum_{j=0}^{n-i} \binom{n-i}{j} 3^{i+j}\\
& = \sum_{i=0}^n \binom{n}{i} 4^i 3^i \sum_{j=0}^{n-i} \binom{n-i}{j} 3^{j}\\
& = \sum_{i=0}^n \binom{n}{i} 12^i4^{n-i}\\
& = O(16^n),
\end{align*}
as desired. To count the number of members of $\mathcal{D}_2(5)$,  suppose first that  $|B| = i$. There are $\binom{n}{i}$ ways to pick $B$. We  pick two subsets of $B$ (which will serve as $B \cap A$ and $B\cap E$) in at most $\left(2^i\right)^2$ ways.  Suppose that $B \cup A$ has size $i + j$. We pick the remaining elements of $A$ in $\binom{n-i}{j}$ ways. We have now completely determined the sets $B$ and $A$, and thus $C\cup D$. There are at most $3^{i+j}$ ways to pick a determining map to pick sets $C$ and $D$. Since we have determined $A$ as well as $C$, we have determined $A \cup C$, as well as $B \cup E$. We also have determined $B$ and $B \cap E$, and so $E$ is determined. While the procedure was different, summing over $i$ and $j$ we have 

\begin{align*}
|\mathcal{D}_2(5)| &\leq \sum_{i=0}^n \binom{n}{i} \left(2^i\right)^2 \sum_{j=0}^{n-i} \binom{n-i}{j} 3^{i+j}\\
& = O(16^n).
\end{align*}

Hence $\mathcal{D}(5) \leq |\mathcal{D}_1(5)|+ |\mathcal{D}_2(5)| = O(16^n)$, as desired.

\end{proof}

The computations for Claim~\ref{del6} and Claim~\ref{del7} are derived similarly.  

By Claims~\ref{del5},~\ref{del6}, and~\ref{del7}, there exist a constant $C$ so that:
\begin{eqnarray}
\Delta &\leq |\mathcal{D}(5)|p^5 + |\mathcal{D}(6)|p^6 + |\mathcal{D}(7)|p^7\nonumber\\
& \leq C (16^n p^5 + 28^n p^6 + 52^n p^7). \label{tozero1}
\end{eqnarray}
Plugging (\ref{tozero1}) into \eqref{JanWuf1}, along with the asymptotics obtained for $|\mathcal{B}|$ in \eqref{Badas}, we have:

\begin{align*}
\p(X = 0) &\leq \left( \Pi_{\{A, B, C, D\}\in \mathcal{B}} \p(I_{A,B,C,D}=0)\right)\exp(\Delta)\\
& \leq (1-p^4)^{\frac{1}{8} 10^n (1+o(1)} \exp \lc C (16^n p^5 + 28^n p^6 + 52^n p^7)\rc\\
& \leq \exp\lc  -\frac{10^n}{8}p^4(1+o(1) + C (16^n p^5 + 28^n p^6 + 52^n p^7)\rc. \numberthis \label{JanFin}
\end{align*}
We remark that 

\[
\left(\frac{1}{10}\right)^{n/4} \ll \left(\frac{1}{52}\right)^{n/7} \ll \left(\frac{1}{28}\right)^{n/6}\ll \left(\frac{1}{16}\right)^{n/5}
\]
and so chosing $\left(\frac{1}{10}\right)^{n/4} \ll p \ll  \left(\frac{1}{52}\right)^{n/7}$ sends the expression in \eqref{JanFin} (and hence $\p(X =0)$) to $0$. As before, we are done by monotonicity.

\end{proof}

\section{Future Directions}  There are many directions for further work.  We list below some of these.
\begin{itemize}
\item New structures for which the ``threshold progressions" idea can be furthered and completed include well-studied areas such as graph connectivity (see the results in \cite{bo}) or cover times for graphs (\cite{pe});
\item There seems to be adequate justification for studying results for coverage of objects in ${\cal A}$ between $s$ and $t$ times by objects in ${\cal C}$; 
\item Section 3: Generalizing Theorem 3.3 to $h\ge 2$ would be of interest, and combating the overlaps between components of sums of two integers would be the primary technical challenge;
\item Section 3: This open problem is related to an original question of Sidon; see, e.g., \cite{o}.  It has been suggested by Kevin O'Bryant.  Sidon's original question was ``How thick can a set $A\subseteq{\mathbb Z}^+$ be if 
\[\sigma(n)=\vert\{(a,b): a+b\in A; a+b=n\}\vert\]
and
\[\delta(n)=\vert\{(a,b): a-b\in A; a-b=n\}\vert\]
satisfy, for each $n$, $\sigma(n)+\delta(n)\le g$."  Note that in this ordered set format, Sidon sets are those for which $\sigma(n)\le 2$ for each $n$.  It is easy to verify that $\sigma(n)\le 2\ {\rm iff}\ \delta(n)\le 1$.  But if $\sigma(n)\le 4$ then it is still possible for $\delta(n)$ to be unbounded.  Sidon's original question has not been the subject of a large-scale investigation.  In our context, however, we might ask for thresholds for the property $\sigma(n)+\delta(n)\le g$.

\item Section 4:  There is a large gap between the thresholds in Theorems 4.1 and 4.6.  Are we asking the right question?
\item Section 5: This section is most in need of development.  What about unions of three or more sets?  Disjoint unions?  Analogous results in the ``at least $g$-sets" genre, to mirror the extremal results in \cite {su}?

\end{itemize}

\section{Acknowledgements}  The research of all four authors was supported by NSF Grant DMS-1263009.

\end{document}